\documentclass{article}
\usepackage{amsmath,amssymb,amsthm,fullpage,cancel,graphicx}

\newcommand{\codomino}{\text{\rm co-domino}}

\newcommand{\F}{\mathcal{F}}
\newcommand{\Forb}{\operatorname{Forb}}
\newcommand{\fourpan}{\text{\rm 4-pan}}

\newcommand{\kttp}{K_{2,3}^+}
\newcommand{\kite}{\text{\rm kite}}
\newcommand{\stool}{\text{\rm stool}}
\renewcommand{\S}{\mathcal{S}}

\newtheorem{thm1}{Theorem}
\newtheorem{cor1}[thm1]{Corollary}

\newtheorem{question}{Question}

\newtheorem{thm}{Theorem}[section]
\newtheorem{lem}[thm]{Lemma}
\newtheorem{cor}[thm]{Corollary}
\newtheorem{obs}[thm]{Observation}

\theoremstyle{definition}

\title{Graphs with the strong Havel--Hakimi property}
\author{Michael D. Barrus$^{1}$ and Grant Molnar$^{2}$}
\begin{document}
\maketitle
\begin{abstract}
The Havel--Hakimi algorithm iteratively reduces the degree sequence of a graph to a list of zeroes. As shown by Favaron, Mah\'{e}o, and Sacl\'{e}, the number of zeroes produced, known as the residue, is a lower bound on the independence number of the graph. We say that a graph has the strong Havel--Hakimi property if in each of its induced subgraphs, deleting any vertex of maximum degree reduces the degree sequence in the same way that the Havel--Hakimi algorithm does. We characterize graphs having this property (which include all threshold and matrogenic graphs) in terms of minimal forbidden induced subgraphs. We further show that for these graphs the residue equals the independence number, and a natural greedy algorithm always produces a maximum independent set.
\end{abstract}
\footnotetext[1]{
Department of Mathematics, University of Rhode Island, Kingston, RI; {\tt barrus@uri.edu}.}
\footnotetext[2]{
Department of Mathematics, Brigham Young University, Provo, UT; {\tt molnar.grant.5772@gmail.com}.}

\section{Introduction}
\label{sec: intro}
Given the degree sequence $d$ of a simple graph, the \emph{Havel--Hakimi algorithm} is the procedure that iteratively removes a largest term $t$ from the sequence and subtracts 1 from the $t$ largest remaining terms, stopping when a list of zeroes is produced. For example, starting with the degree sequence $d = (3,2,2,2,2,1)$, the algorithm produces the sequences below.
\begin{align*}
d=d^0&: (3,2,2,2,2,1)\\
d^1&: (2,1,1,1,1)\\
d^2&: (1,1,0,0)\\
d^3&: (0,0,0)
\end{align*}

Here $d^i$ indicates the sequence obtained after $i$ iterations of the removal/reduction step. Throughout the paper, we will order the terms of each $d^i$ from largest to smallest.

This algorithm was introduced independently by Havel~\cite{Havel55} and Hakimi~\cite{Hakimi62} as a means for testing whether a list of nonnegative integers $d$ is the degree sequence of a graph; $d$ is a degree sequence if and only if the algorithm terminates at a list of zeroes. For a degree sequence $d$, the \emph{residue} of $d$, denoted $R(d)$, is the number of zeroes that remain. For example, the steps above show that $R((3,2,2,2,2,1)) = 3$. For any graph $G$ having degree sequence $d$, we define $R(G)=R(d)$ and speak interchangeably of the residue of the graph and the residue of the degree sequence.

Let us say that a vertex $v$ in a graph $G$ has \emph{the Havel--Hakimi property} if $v$ is a vertex of maximum degree in $G$ having neighbors with as large of degrees as possible; more precisely, $v$ has the Havel--Hakimi property if it has maximum degree in $G$ and no vertex adjacent to $v$ has degree smaller than any vertex not adjacent to $v$. Each reduction from $d^i$ to $d^{i+1}$ in the Havel--Hakimi algorithm simulates the effect on the degree sequence of deleting a vertex with the Havel--Hakimi property.

Not every realization of a degree sequence has such a vertex. For example, there is no vertex with the Havel--Hakimi property in the realization of $(3,2,2,2,2,1)$ obtained by attaching a pendant vertex to a vertex of a chordless 5-cycle. However, as shown in most proofs of the result of Havel and Hakimi (see, for instance,~\cite[Theorem 1.12]{ChartrandLesniakZhang11}), given any degree sequence $d$ and any vertex $v$ chosen to have maximum degree, there is a realization of $d$ where $v$ has the Havel--Hakimi property.

We say that a graph $G$ has the \emph{strong Havel--Hakimi property} if in every induced subgraph $H$ of $G$, \emph{every} vertex of maximum degree has the Havel--Hakimi property. Let $\mathcal{S}$ denote the class of graphs with this property. In this paper we study $\S$, characterizing its elements and describing some properties of these graphs.

First, we characterize $\S$ in terms of its minimal forbidden induced subgraphs. Given a class $\F$ of graphs, let $\Forb(\F)$ denote the class of all graphs having no induced subgraph isomorphic to an element of $\F$. Before presenting the set $\F$ for which $\Forb(\F)=\S$, we need some definitions.

Given graphs $G$ and $H$, let $G+H$ denote the disjoint union of $G$ and $H$, and express the disjoint union of $m$ copies of $G$ by $mG$. Given positive integers $m$ and $n$, let $P_n$ denote the path with $n$ vertices, let $K_n$ be the complete graph on $n$ vertices, and let $K_{m,n}$ be the complete bipartite graph with partite sets of sizes $m$ and $n$. Let $\kttp$ be the graph obtained by adding an edge to the larger partite set of $K_{2,3}$; this graph is the complement of $K_2+P_3$. We define the \emph{4-pan}, \emph{kite}, \emph{stool}, and \emph{co-domino}, respectively, to be the connected graphs shown in Figure~\ref{fig: S forb graphs}.
\begin{figure}
\centering
\includegraphics{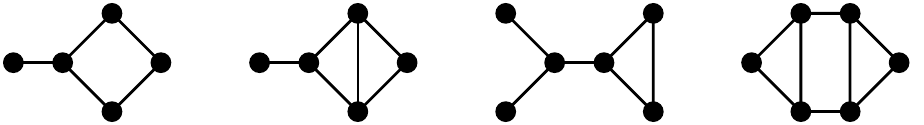}
\caption{The 4-pan, kite, stool, and co-domino graphs.}
\label{fig: S forb graphs}
\end{figure}

\begin{thm1}\label{thm: Forb for S}
$\S= \Forb(\{P_5,\fourpan,K_{2,3},\kttp,\kite,2P_3,P_3+K_3,\stool,\codomino\})$.
\end{thm1}

We prove this theorem in Section~\ref{sec: forb for S}. Along the way, we see how $\S$ compares with some other well known graph families. (A discussion of these families accompanies the proof.)

\begin{cor1} \label{cor: S and classes}
The class $\S$ contains all matrogenic graphs (and hence all matroidal and threshold graphs as well).
\end{cor1}

Further motivation for studying $\S$ comes from the connection between the Havel--Hakimi residue and the independence number of a graph. It follows from the definition of the strong Havel--Hakimi property that when we iteratively delete vertices of maximum degree from a graph $G$ in $\S$, we arrive at an edgeless induced subgraph of $G$ with $R(G)$ vertices. Thus $R(G) \leq \alpha(G)$ for graphs $G$ in $\S$, where $\alpha(G)$ denotes the independence number of $G$. Favaron et al.~\cite{FavaronEtAl91} showed that this inequality is in fact true for all graphs.

\begin{thm1}[\cite{FavaronEtAl91}; see also \cite{GriggsKleitman90,Triesch96}]
\label{thm: R leq alpha}
For every graph $G$, we have $R(G) \leq \alpha(G)$.
\end{thm1}

Given that the residue of a graph $G$ may be computed in $O(e)$ steps, where $e$ is the number of edges in the graph, while determining $\alpha(G)$ is NP-hard, we may wish to know how well $R(G)$ approximates $\alpha(G)$. In particular, it is an open problem to find a nontrivial answer to the following question.

\begin{question}\label{question}
\text{For which graphs are the residue and independence number equal?}
\end{question}

In considering when the residue of a graph equals its independence number, it seems natural to consider procedures that generate independent sets by removing high-degree vertices. The simplest such heuristic, which has been called Maxine~\cite{GriggsKleitman90}, begins with a graph $G$ and, as long as the current graph contains any edges, iteratively chooses and deletes a vertex of maximum degree. The vertices that remain at the completion of the process necessarily form an independent set in $G$, and the size $M$ of this set is clearly a lower bound on $\alpha(G)$. We note that different choices of a vertex of maximum degree can result in different sizes of independent sets; performing Maxine on the 5-vertex path can yield an independent set of size $M$ equal to either 2 (which equals the residue) or 3 (which equals the independence number), depending on which vertex is deleted first. Still, Griggs and Kleitman proved the following:

\begin{thm1}[\cite{GriggsKleitman90}] If $M$ is the size of the independent set produced by any application of the Maxine heuristic, then $R(G)\leq M \leq \alpha(G)$.
\end{thm1}

Thus if the residue equals the independence number, the residue must equals the size of the independent set generated by any application of Maxine. As we now see, this is true for all graphs in $\S$.

\begin{thm1}\label{thm: R equals alpha for S}
If a graph $G$ is in $\S$, then $R(G)=\alpha(G)$. Hence the Maxine heuristic always produces a maximum independent set for graphs in $\S$.
\end{thm1}

A proof of this theorem is given in Section~\ref{sec: S residues}.

Throughout the paper we use $\vee$ to denote the join of graphs, and for $n \geq 1$, we use $C_n$ to denote the cycle with $n$ vertices. For a collection $J$ of vertices in $G$, we use $G[J]$ to denote the induced subgraph of $G$ with vertex set $J$, and we use $G-J$ to indicate the graph obtained by deleting all vertices in $J$ from $G$.

\section{Proofs of Theorem~\ref{thm: Forb for S} and Corollary~\ref{cor: S and classes}}
\label{sec: forb for S}

This section contains a proof of Theorem~\ref{thm: Forb for S}, showing that the set \[\F_\S=\{P_5,\fourpan,K_{2,3},\kttp,\kite,2P_3,P_3+K_3,\stool,\codomino\};\] contains all minimal forbidden induced subgraphs for the class $\S$. As a byproduct, partway through the section we obtain a proof of Corollary~\ref{cor: S and classes}, showing that $\S$  contains all matrogenic graphs. We begin by showing that $\S \subseteq \Forb(\F_\S)$.

\begin{lem}\label{lem: S is in Forb(F)}
Every graph in $\S$ is an element of $\Forb(\F_{\S})$.
\end{lem}
\begin{proof}
Observe that in each graph in $\F_{\S}$, some vertex of maximum degree has a neighbor $w$ and a non-neighbor $u$ such that the degree of $w$ is less than the degree of $u$. By definition, such vertices cannot occur in any induced subgraph of an element of $\S$; thus graphs in $\S$ must be $\F_{\S}$-free.
\end{proof}

We now show that $\Forb(\F_{\S}) \subseteq \S$. Suppose to the contrary that $\Forb(\F_{\S})$ contains an element $G$ that is not in $\S$. By definition, some induced subgraph $G'$ of $G$ contains a vertex of maximum degree not having the Havel--Hakimi property. Since $\Forb(\F_{\S})$ is a hereditary class and by definition $G' \notin \S$, it suffices to assume that $G'=G$. Some vertex $v$ of maximum degree in $G$ has a neighbor $w$ and a nonneighbor $u$ such that $d_G(w)<d_G(u)$. Because $v$ is a neighbor of $w$ that $u$ does not share, $u$ must be adjacent to two vertices to which $w$ is not. We illustrate the adjacency relationships among, $u$, $v$, $w$, and these other vertices in Figure~\ref{fig: matro forb config}, where dotted segments indicate non-adjacencies.
\begin{figure}
\centering
\includegraphics{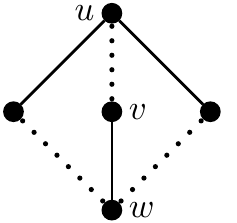}
\caption{A configuration present when a vertex of maximum degree lacks the Havel--Hakimi property.}
\label{fig: matro forb config}
\end{figure}

As a brief aside, we remark that the configuration illustrated in Figure~\ref{fig: matro forb config} is precisely the forbidden configuration that characterizes the matrogenic graphs, as shown by F\"{o}ldes and Hammer in~\cite{FoldesHammer78}. The matrogenic graphs were introduced in~\cite{FoldesHammer78} as those graphs for which the vertex sets of a given 4-vertex configuration are the circuits of a matroid with ground set $V(G)$. The matrogenic graphs are known to properly contain such graph classes as the matroidal graphs and the threshold graphs (see~\cite{MahadevPeled95} for a survey on all of these graph classes). Note that the forbidden configuration characterization of matrogenic graphs implies that the class is closed under taking induced subgraphs. We thus conclude that in every matrogenic graph each vertex of maximum degree has the Havel-Hakimi property. Thus $\S$ contains the class of matrogenic graphs, as stated in Corollary~\ref{cor: S and classes}.

Continuing with our proof of Theorem~\ref{thm: Forb for S}, we use the assumption that $G$ is $\F_{\S}$-free to arrive at a contradiction in every case.

\bigskip
\noindent \textsc{Claim:} $u$ and $v$ have a common neighbor that $w$ does not.

\medskip
Assume to the contrary that every common neighbor of $u$ and $v$ is adjacent to $w$. Let $x$ and $y$ be vertices adjacent to $u$ and not to $w$. By assumption, neither $x$ nor $y$ is adjacent to $v$. Now since $d_G(v) \geq d_G(u)>d_G(w)$, $v$ must have some neighbor $z$ that is not adjacent to $w$. Note that $z$ is distinct from $u$, $w$, $x$, and $y$. By assumption, $z$ is not adjacent to $u$.

Suppose that $x$ is not adjacent to $y$. Since $G[\{u,v,x,y,z\}]$ is not isomorphic to $P_5$ or the 4-pan, $z$ is adjacent to neither of $x$ or $y$. However, this would imply that $G[\{u,v,w,x,y,z\}]$ is isomorphic to $2P_3$ or that $G[\{u,v,w,x,z\}]$ is isomorphic to $P_5$, depending on whether $u$ is adjacent to $w$.

Suppose instead that $x$ is adjacent to $y$. Now since $G[\{u,v,x,y,z\}]$ is not isomorphic to the kite graph, $z$ is nonadjacent to at least one of $x$ and $y$; by symmetry, assume $zx$ is not an edge of $G$. If $uw$ is an edge, then $G[\{u,v,w,x,z\}]$ is isomorphic to $P_5$, a contradiction, so $u$ is not adjacent to $w$. If $yz$ is an edge, then $G[\{v,w,x,y,z\}]$ is isomorphic to $P_5$, again a contradiction, so $y$ is not adjacent to $z$. However, now we see that $G[\{u,v,w,x,y,z\}]$ is isomorphic to the $P_3+K_3$, a contradiction. Having arrived at a contradiction in every case, our proof of the claim is complete.

\bigskip
We now establish a lemma that holds for all $\F_{\S}$-free graphs.

\bigskip
\noindent \textsc{Claim:} Consider an $\F_{\S}$-free graph $H$ in which the subgraph induced on vertices $a,b,c,d$ contains edges $ab$, $bc$, $cd$, and non-edges $ac$ and $bd$. If some vertex $p$ is adjacent to $a$ but not to $d$, then $p$ is adjacent to $b$ and not to $c$.

\medskip
Assume that $p$ is adjacent to $a$ but not to $d$. If $p$ is not adjacent to $b$, then $H[\{a,b,c,d,p\}]$ is isomorphic to $P_5$, $\fourpan$, or $K_{2,3}$. If $p$ is adjacent to both $b$ and $c$, then $H[\{a,b,c,d,p\}]$ is isomorphic to either $\kite$ or $\kttp$, which proves the claim.

\bigskip
By the first claim above, $G$ has a vertex $x$ that is adjacent to $u$ and $v$ and not to $w$. Let
\begin{align*}
N_u &= \{p \in V(G)-\{u,v,w,x\}: pu \in E(G), pw \notin E(G)\};\\
N_v &= \{p \in V(G)-\{u,v,w,x\}: pv \in E(G), px \notin E(G)\};\\
N_w &= \{p \in V(G)-\{u,v,w,x\}: pw \in E(G), pu \notin E(G)\};\\
N_x &= \{p \in V(G)-\{u,v,w,x\}: px \in E(G), pv \notin E(G)\}.
\end{align*}

Since $G$ is $\F_{\S}$-free, we may apply the second claim twice with $\{a,b,c,d\} =\{u,v,w,x\}$ to conclude that $N_w \subseteq N_v$ and $N_u \subseteq N_x$.

\bigskip
\noindent \textsc{Claim:} If $y \in N_u$ and $z \in N_v$, then $G[\{u,v,w,x,y,z\}]$ is one of the graphs shown in Figure~\ref{fig: dumbbell}.
\begin{figure}
\centering
\includegraphics{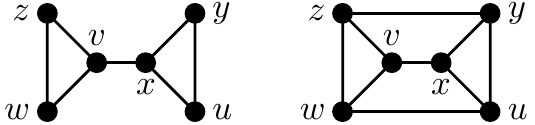}
\caption{The graphs from Lemma~\ref{lem: dumbbell}.}
\label{fig: dumbbell}
\end{figure}
\medskip
Since $N_u \subseteq N_x$, vertex $y$ is adjacent to $x$ and not adjacent to $v$.

Suppose first that $u$ is adjacent to neither $w$ nor $z$. If $y$ is adjacent to $z$, then since $G[\{v,w,x,y,z\}]$ is not isomorphic to $\fourpan$, we have that $wz$ is an edge; however, then $G[\{u,v,w,x,y,z\}]$ is isomorphic to $\codomino$, a contradiction. Thus $yz$ is not an edge. Now since $G[\{u,v,w,x,y,z\}]$ is not isomorphic to the stool graph, $wz$ is an edge, and $G[\{u,v,w,x,y,z\}]$ is equal to the first graph shown in Figure~\ref{fig: dumbbell}.

Suppose instead that $u$ has a neighbor in $\{w,z\}$. Without loss of generality assume $uw$ is an edge. Since $G[\{u,v,w,x,z\}]$ is isomorphic to none of $\fourpan, K_{2,3},\kttp$, we see that $uz$ is not an edge and $wz$ is an edge in $G$. Now since $G[\{u,v,w,x,y,z\}]$ is not isomorphic to $\codomino$, $yz$ must be an edge, and $G[\{u,v,w,x,y,z\}]$ is the second graph shown in Figure~\ref{fig: dumbbell}. Thus the claim is true.

\bigskip
Inspecting the neighbors of $z$ in both graphs in Figure~\ref{fig: dumbbell}, we conclude that if $N_u$ is nonempty, then $N_v = N_w$. Now since $d_G(u)>d_G(w)$, certainly $N_u$ is nonempty. However, recalling our assumptions on the degrees of $u$, $v$, and $w$, and that $N_u \subseteq N_x$, we have \[|N_w| < |N_u| \leq |N_x| \leq |N_v| = |N_w|.\] This  contradiction implies that $\Forb(\F_\S) \subseteq \S$, which completes our proof of Theorem~\ref{thm: Forb for S}.

\section{Proof of Theorem~\ref{thm: R equals alpha for S}}
\label{sec: S residues}

In this section we prove Theorem~\ref{thm: R equals alpha for S} from the Introduction. We preface our main argument with results on the Maxine heuristic. First is the observation made there about Maxine and the residue, reproduced here for easy reference.

\begin{obs}\label{obs: R=M}
If $G$ is a graph and at each step of an application of the Maxine heuristic to $G$ the maximum-degree vertex to be deleted has the Havel--Hakimi property, then the independent set returned by Maxine has size $R(G)$. Consequently, this conclusion holds for all graphs in $\S$.
\end{obs}

\begin{lem}\label{lem: P5 or C4}
Let $G$ be a graph having at least one edge. If a vertex $v$ of maximum degree in $G$ belongs to every maximum independent set in $G$, then $v$ either belongs to an induced 4-cycle in $G$, or $v$ is the center vertex of an induced $P_5$.
\end{lem}
\begin{proof}
Suppose that $v$ is a vertex of maximum degree vertex in $G$ and that $v$ belongs to every maximum independent set in $G$. Let $A$ be a set obtained by deleting $v$ from a maximum independent set in $G$, and let $B = V(G)-\{v\}-A$. By definition, $|A|=\alpha(G)-1$, and the neighbors of $v$ all belong to $B$. Furthermore, since $v$ belongs to every maximum independent set, every vertex of $B$ has to have at least one neighbor in $A$. It follows that some vertex of $B$ has degree at least 2, and hence that $v$ does as well.

If the neighbors of $v$ are all pairwise adjacent, then each has a degree larger than that of $v$, a contradiction. Thus $v$ has two neighbors $w$ and $x$ that are nonadjacent. If $w$ and $x$ have a common neighbor in $A$, then this vertex and $v,w,x$ induce a 4-cycle in $G$. If $w$ and $x$ have no common neighbor, then, since each has a neighbor in $A$, these neighbors and $v,w,x$ induce a copy of $P_5$ in $G$ with $v$ as the center vertex.
\end{proof}

\begin{cor}\label{cor: C4 P5 free}
The Maxine heuristic always produces a maximum independent set when applied to a $\{C_4,P_5\}$-free graph.
\end{cor}
\begin{proof}
It follows from Lemma~\ref{lem: P5 or C4} that deleting a vertex of maximum degree in a $\{C_4,P_5\}$-free graph $G$ cannot reduce the independence number of $G$. Inductively applying this result (since the resulting subgraphs are also $\{C_4,P_5\}$-free), we conclude that the independent set resulting from the Maxine algorithm has size $\alpha(G)$.
\end{proof}

We now complete our proof of Theorem~\ref{thm: R equals alpha for S}. In the following, let $G$ be an element of $\S$. We recall from Theorem~\ref{thm: Forb for S} that $G$ is $\{P_5, \fourpan, K_{2,3}, \kttp, \kite, 2P_3, P_3+K_3, \stool, \codomino\}$-free, and we will refer to these forbidden subgraphs often in what follows.

We proceed by induction on the number of vertices in $G$. One easily verifies that $R(G)=\alpha(G)$ for all graphs on four or fewer vertices, and we suppose that $G$ has $n$ vertices and that $R(H)=\alpha(H)$ for all graphs $H$ in $\S$ with fewer than $n$ vertices.

Let $v$ be a vertex of maximum degree in $G$. Since $G$ belongs to $\S$, the vertex $v$ has the Havel--Hakimi property, and $G-v$ belongs to $\S$. Using the induction hypothesis, we have $R(G)=R(G-v)=\alpha(G-v)$. If $\alpha(G-v) = \alpha(G)$ our proof is complete, so suppose henceforward that $v$ belongs to every maximum independent set in $G$. Since $G$ is $P_5$-free, Lemma~\ref{lem: P5 or C4} implies that $v$ belongs to an induced $4$-cycle in $G$. Let $W$ be the vertex set of this cycle, and let $U$ be the set of all vertices of $G$ that are adjacent to every vertex of $W$.

Since $G$ is $\{K_{2,3},\kttp\}$-free, $G[U]$ is $\{3K_1,K_2+K_1\}$-free. It is an easy exercise to show that a graph is complete multipartite if and only if it is $\{K_2+K_1\}$-free. Thus $G[U]$, and by extension $G[U \cup W]$ is a complete multipartite graph where the maximum size of a partite set is $2$. Let $B$ be the union of all partite sets of size $1$, and let $C$ be the union of all partite sets of size $2$ (so $W \subseteq C$). Finally, let $A$ be the subset of $V(G)-B-C$ consisting of all vertices not having any neighbors in $C$, and let $D = V(G)-A-B-C$.

Any four vertices comprising two partite sets in $G[C]$ induce $C_4$, and since $G$ is $\{\fourpan,K_{2,3},\kttp\}$-free, each vertex in $D$ is adjacent to either no vertex, all vertices, or exactly two consecutive vertices of such a 4-cycle. By definition each vertex in $D$ has both a neighbor and a nonneighbor in $W$; it follows that vertices in $D$ each have exactly one neighbor in each partite set of $G[C]$.

Since $G$ does not induce $P_5$ in $G[W \cup \{a,d\}]$ for any vertices $a \in A$ and $d \in D$, no edges exist joining vertices in $A$ with vertices in $D$. Likewise, since $G$ does not induce the kite graph in $G[W \cup \{b,d\}]$ for any vertices $b \in B$ and $d \in D$, all edges possible between $B$ and $D$ must exist in $G$. Now any vertex in $B$ has degree at least $|B|+|C|+|D|-1$, while any vertex of $C$ has degree at most $|B|+|C|+|D|-2$. Since $v$ belongs to $C$ and is a vertex of maximum degree in $G$, we see that $B = \emptyset$.

Let $u$ be the other vertex in the partite set of $G[C]$ that contains $v$. Every maximum independent set in $G$ must contain a vertex from $D$ that is adjacent to $u$; otherwise, we could replace $v$ by $u$ in the set and obtain a maximum independent set in $G$ that does not contain $v$, a contradiction. Thus $D$ is nonempty, and we conclude that $G[C]$ is isomorphic to $C_4$, since if a vertex $d$ of $D$ were adjacent to exactly one vertex of each of three different partite sets in $G[C]$, then the induced subgraph with vertex set consisting of $d$, one of the partite sets, and the non-neighbors of $d$ in the other two partite sets is isomorphic to the kite, a contradiction.

Partition $D$ into two sets $D_u$, $D_v$ according to which of $u$ and $v$ each vertex is adjacent to. As observed above, $D_u$ is nonempty. Since $v$ has maximum degree in $G$, we see that $|D_v|\geq |D_u| \geq 1$. Let $x \in D_u$ and $y \in D_v$.

Now let $\{s,t\}$ be the partite set different from $\{u,v\}$ in $G[C]$. Each of $x$ and $y$ is adjacent to exactly one of $s$ and $t$. Without loss of generality, assume that $x$ is adjacent to $s$.

If $y$ is also adjacent to $s$, then since $G[\{t,u,v,x,y\}]$ is not isomorphic to $P_5$, vertices $x$ and $y$ must be adjacent. However, then $G[\{s,t,u,x,y\}]$ is isomorphic to the kite graph, a contradiction. Thus $y$ is adjacent to $t$. By similar arguments, all vertices of $D_u$ are adjacent to $s$, and all vertices of $D_v$ are adjacent to $t$.

Each vertex in $D_u$ must be adjacent to every vertex of $D_v$, since $G$ is $\codomino$-free. Furthermore, $D$ contains no vertex other than $x$ or $y$, since a third vertex, together with $x$, $y$, $u$, and $v$, would form the vertex set of an induced 4-pan or kite. Thus $G[C \cup D]$ is isomorphic to the graph on the right in Figure~\ref{fig: dumbbell}, which contradicts the claim that $v$ belongs to every maximum independent set in $G$, since a maximum independent in $G$ will include precisely two vertices from $G[C \cup D]$, and any two nonadjacent vertices will do, as there are no edges between $A$ and $D$. We conclude that $v$ cannot belong to every maximum independent set, and in fact $R(G)=R(G-v)=\alpha(G-v)=\alpha(G)$, which completes our induction.


\begin{thebibliography}{99}

\bibitem{ChartrandLesniakZhang11} G.~Chartrand, L.~Lesniak, and P.~Zhang, Graphs \& digraphs (fifth ed.), CRC Press, Boca Raton, FL, 2011.

\bibitem{FavaronEtAl91} O.~Favaron, M.~Mah\'{e}o, and J.-F.~Sacl\'{e}, On the residue of a graph, J.~Graph Theory 15 (1) (1991) 39--64.

\bibitem{FoldesHammer78} S.~F\"{o}ldes and P.L.~Hammer, On a class of matroid-producing graphs, Combinatorics (Proc.~Fifth Hungarian Colloq., Keszthely, 1976), Vol.~I, pp.~331--352, 
Colloq.~Math.~Soc.~J\'{a}nos Bolyai, 18, North-Holland, Amsterdam-New York, 1978. 

\bibitem{GriggsKleitman90} J.~Griggs and D.J.~Kleitman, Independence and the Havel-Hakimi residue, Graph theory and applications (Hakone, 1990), Discrete Math.~127 (1994), no.~1--3, 209--212.

\bibitem{Havel55} V.~Havel, A remark on the existence of finite graphs (Hungarian), \u{C}asopis P\u{e}st Mat.~80 (1955), 477--480.

\bibitem{Hakimi62} S.~Hakimi, On the realizability of a set of integers as degree sequences of
the vertices of a graph, J.~SIAM Appl.~Math.~10 (1962), 496--506.

\bibitem{MahadevPeled95} N.V.R.~Mahadev and U.N.~Peled, Threshold graphs and related topics, Annals of Discrete Mathematics, 56. North-Holland Publishing Co., Amsterdam, 1995.

\bibitem{Triesch96} E.~Triesch, Degree sequences of graphs and dominance order, J.~Graph Theory 22 (1996), no.~1, 89--93.
\end{thebibliography}
\end{document}